\documentclass[11pt]{amsart}
\usepackage[centertags]{amsmath}
\usepackage{amsfonts}
\usepackage{amssymb}
\usepackage{amsthm}
\usepackage{amscd}
\usepackage[all]{xy}
\usepackage{url} 

\newtheorem{lem}{Lemma}
\newtheorem{cor}[lem]{Corollary}

\newtheorem{thm}[lem]{Theorem}

\newcommand{\ra}{\rightarrow}

\newcommand{\GL}{\operatorname{GL}}
\newcommand{\Spin}{\operatorname{Spin}}
\newcommand{\SO}{\operatorname{SO}}

\newcommand{\cO}{{\mathcal O}}
\newcommand{\cP}{{\mathcal P}}

\newcommand{\cX}{{\mathcal X}}
\newcommand{\cY}{{\mathcal Y}}

\newcommand{\bR}{{\mathbb R}}
\newcommand{\bC}{{\mathbb C}}
\newcommand{\bZ}{{\mathbb Z}}
\newcommand{\bQ}{{\mathbb Q}}
\newcommand{\bF}{{\mathbb F}}

\begin{document}
\title{Profinite completions and Kazhdan's property (T)}
\author{Menny Aka}
\address{Department of Mathematics,\newline%
\indent
The Hebrew University,\newline%
\indent
Givat Ram, Jerusalem 91904, Israel.}
\email{mennyaka@math.huji.ac.il}


\begin{abstract}
We show that property $(T)$ is not profinite, that is, we construct two finitely generated residually finite groups which have isomorphic profinite completions while  one admits property (T) and the other does not. This settles a question raised by M. Kassabov.
\end{abstract}
\maketitle
\bibliographystyle{alpha}
\section{Introduction}
Two finitely generated groups have isomorphic profinite completions if and only if they have the same collection of finite quotients.
A property $\cP$ of finitely generated residually finite groups is called a \emph{profinite property} if the following is satisfied: if $\Gamma_1$ and $\Gamma_2$ are such groups with $\widehat{\Gamma_1} \cong \widehat{\Gamma_2}$ (i.e., the profinite completion of $\Gamma_1$ and $\Gamma_2$ are isomorphic) then $\Gamma_1$ has $\cP$ if and only if $\Gamma_2$ has $\cP$.

There are various interesting properties which are trivially profinite properties: e.g., having infinite Abelanization, rate of subgroup growth etc. But there are other profinite properties which are in some sense less trivial, for example,  having polynomial word growth (this follows from Gromov's seminal result on polynomial growth) and for a finitely presented groups the property of being large (a group is said to be \emph{large} if one of its finite index subgroups has a free non-Abelian quotient)\cite{Lack07}. On the other hand, various other properties are not profinite. In a recent work, it is shown  in \cite{CH09} that there exist two finitely presented residually-finite groups such that one is conjugacy separable and the other is not, and yet they have isomorphic profinite completions; this shows that being conjugacy separable is not a profinite property. In a work in preparation, Kassabov showed that it follows from \cite{Kass07},\cite{KN06}  that property $(\tau)$ is not a profinite property, and he asked whether Kazhdan property (T) is a profinite property. In this note, we show that Kazhdan property (T) is not a profinite property. Explicitly we prove:
\begin{thm}\label{thm:main theorem}
Let $D$ be a positive square-free integer, $k:= \bQ(\sqrt{D})$ and $\cO_k$ its ring of integers. Fix an integer $n \geq 6$ and let $\Gamma = \Spin(1,n)(\cO_k)$ and $\Lambda = \Spin(5,n-4)(\cO_k)$. Then, there exist finite-index subgroups $\Gamma_0< \Gamma$ and $\Lambda_0<\Lambda$ such that the profinite completion of $\Gamma_0$ is isomorphic to the profinite completion of $\Lambda_0$ while $\Lambda_0$ admits property $(T)$ and $\Gamma_0$ does not. Therefore, Kazhdan property (T) is not profinite.
\end{thm}
In particular, there exist non-isomorphic arithmetic groups with isomorphic profinite completions. Note that $\Spin(1,n)(\cO_k)$ (resp.\ $\Spin(5,n-4)(\cO_k)$) are central extensions of irreducible lattices in $H_1:=\SO(1,n)\times \SO(1,n)$ (resp.\ $H_2:=\SO(5,n-4)\times \SO(5,n-4)$). We can therefore deduce that $H_1$ and $H_2$ have lattices $\Gamma_1$ and $\Lambda_1$ resp.\ with isomorphic profinite completions. As $\bR-rank(H_1)=2$ while $\bR-rank(H_2)=10$ we see that the rank of the ambient Lie group is also not a profinite property (see also Remark \ref{rank remark} and compare with \cite{PrRaGu72}.)

While we can construct many non-trivial examples of non-isomorphic arithmetic groups with isomorphic profinite completions as the one above, we show in \cite{Aka} that every set of higher rank arithmetic groups with isomorphic profinite completion consists of \emph{finitely many} isomorphism classes.

Our construction is based on a very simple idea. Vaguely speaking, the groups above have the congruence subgroup property and hence their profinite completions are essentially a product of the $p$-adic completions. Their $p$-adic completions agree since the quadratic forms $(1,n)$ and $(5,n-4)$ agree in every local field. Nevertheless, they do not agree over $\bR$. Moreover, $\Gamma_0$ is a lattice in $\Spin(1,n)\times \Spin(1,n)$ while $\Lambda_0$ is a lattice in $\Spin(5,n-4)\times \Spin(5,n-4)$. The latter has $(T)$ while the former does not.
We also note that both groups admit property tau.


This note is organized as follows: Much of our construction relies on quadratic forms theory, their Clifford algebra and their spin groups, so we review the relevant results and definitions in \S2 together  with a basic Lemma on profinite completions. In \S3 we prove Theorem \ref{thm:main theorem}.\\ 
\textbf{Acknowledgments:} I would like to thank my advisor, Alex Lubotzky, for suggesting the problem and for reading several manuscripts of this note. I would also like to thank Andrei Rapinchuk and Uzi Vishne for help and references on quadratic forms theory and their spin groups. Finally, I'd like to thank Martin Kassabov and the referee for their valuable comments.
This note is a part of the author's Ph.D. thesis at the Hebrew University. This research was partially funded by the "Hoffman Leadership and Responsibility" fellowship program, at the Hebrew University. I also thank the ISEF and the ERC for their support.

\section{Preliminaries}
\subsection{Quadratic forms and their spin groups}
We denote the ring of $p$-adic integers by $\bZ_p$.
\begin{lem}
The quadratic forms $$q_1:= x_1^2+x_2^2+x_3^2+x_4^2,\quad  and \quad q_2: = -q_1 = -x_1^2-x_2^2-x_3^2-x_4^2$$ are equivalent over $\bZ_p$ for every prime $p$, that is, for each prime $p$ there exist $M_p\in \GL_4(\bZ_p)$ such that $M_pM_p^t=-I$.
\end{lem}
\begin{proof}
We first note that for any $p$, the equation 
\begin{equation}
x_1^2+x_2^2+x_3^2+x_4^2=-1
\label{eq:minus 1}
\end{equation}  has solutions $(x_p,y_p,z_p,w_p)\in \bZ_p^4$. Indeed, for $p\neq 2$, this follows from an easy application of Hensel's lemma together with the fact that in $\bF_p$ each element is a sum to two squares. For $p=2$, we may take  $(x_2,y_2,z_2,w_2)=(2 , 1 , 1 , \sqrt{-7})$. Now, for each $p$ consider the matrix
$$M_p:= \left(
\begin{array}{cccc}
x_p & y_p & z_p & w_p \\
-y_p & x_p & -w_p & z_p \\
-z_p & w_p & x_p & -y_p \\
-w_p & -z_p & y_p & x_p 
\end{array} \right)
$$
and note that $M_p\in M_4(\bZ_p)$, $M_pM_p^t=-I$. This shows that $q_1$ and $q_2$ are equivalent over $\bZ_p$ for all primes $p$.
\end{proof}

\begin{cor}\label{cor:quadratics equivalence}
Let $q_{m,n}= \sum_{i=1}^m x_i^2 - \sum_{i=1}^n y_i^2$. For $m\geq4$, the quadratic form $q_{m,n}$ and $q_{m-4,n+4}$ are integrally equivalent over $\bZ_p$ for all primes $p$. It follows that for any number field $k$ and a finite place $v$,  we have that  $q_{m,n}$ and $q_{m-4,n+4}$ are equivalent over $\cO_v$, where $\cO_v$ is the ring of integers of $k_v$.
\end{cor}

\begin{proof}
Let $V_i$ be the quadratic space associated to $q_i,\, i=1,2$, and $V_{m,n}$ the quadratic space associated to $q_{m,n}$. Then for $m\geq4$, $$V_{m,n} \cong V_1\oplus V_{m-4,n}\quad \text{and} \quad V_{m-4,n+4}\cong V_2\oplus V_{m-4,n}$$ as quadratic spaces. As $V_1\cong V_2$ over $\bZ_p$ for all $p$, so does $V_{m,n}$ and $V_{m-4,n+4}$.
\end{proof}

In order to fix our notation and to fix specific representations of the groups involved, we recall the relevant definitions and give an explicit definition of the Clifford algebra and the spin group of a quadratic form. All objects and results that are described here can be found in \cite{cassels:RQF} (in particular, chapter 10, \S2) or in \cite[\S 20]{FulHar91}.

Given a non-degenerate quadratic space $(V,\phi)$ of dimension $n$ over a field $k$, there exists an associative algebra $C(V,\phi)$ over $k$ that contains $V$ as a linear subspace and satisfies:
\begin{enumerate}
	\item $C(V,\phi)$ is of dimension $2^n$ as a vector space.
	\item For all $x\in V$, $x\cdot x=\phi(x)$.
	\item $C(V,\phi)$ is generated as an algebra by $V$.
\end{enumerate}
Moreover, these properties determine $C(V,\phi)$ uniquely (up-to a $k$-algebra isomorphism that fix $V$), and it is called the Clifford algebra of $V$.

We now give an explicit description of $C(V,\phi)$. This description depends on a choice of a normal basis, that is, a basis which is orthogonal with respect to the bilinear form associated to the quadratic form (see \cite[Chapter 10 \S2]{cassels:RQF}). 
Every non-degenerate quadratic space admits a normal basis and we fix a normal basis $e_1,\ldots,e_n$ for $(V,\phi)$. Recall that $C(V,\phi)$ contains $V$ and let $J$ be any subset of $\{1,2,\ldots,n\}$ arranged in ascending order, say $$j_1<j_2<\ldots<j_r $$ where $r\leq n$, and let $$e(J):= e_{j_1}\cdot e_{j_2}\cdot \ldots \cdot e_{j_r},$$ be the multiplication of the $e_j$'s and let $e(\emptyset)$ be the unit element. Then, the set $$\{e(J): J\subset \{1,2,\ldots,n\} \, \text{ ordered in ascending order} , 0\leq|J|\leq n\}$$  is a basis of $C(V,\phi)$. As $C(V,\phi)$ is also generated by $V$, the multiplication in $C(V,\phi)$ is determined by $$e_ie_i = \phi(e_i)$$ and for $i\neq j$,   $$e_ie_j = -e_je_i.$$

The even Clifford algebra of $C(V,\phi)$ is the algebra generated by $\{e(J): |J|\, \text{is even}\}$, and is denoted by $C^0(V,\phi)$. There exists an involution $$':C(V,\phi) \ra C(V,\phi)$$ defined on the basis elements by $(e(J))' = e_{j_r}\cdot e_{j_{r-1}}\cdot \ldots \cdot e_{j_1}$ when $J$ is as above, and extended linearly. Let $(C^0(V,\phi))^*$ denote the group of invertible elements of $C^0(V,\phi)$ and define $\Spin(V,\phi)$ as $$\Spin(V,\phi) := \{x\in (C^0(V,\phi))^*: xx'=1,\, x V x'\subseteq V\}.$$
The group $\Spin(V,\phi)$ admits a faithful irreducible linear representation to $\GL(C^0(V,\phi))$, by acting as right multiplication. We endow $\GL(C^0(V,\phi))$ with the structure that is induced from the basis $\{e(J): |J|\, \text{is even}\}$ of  $C^0(V,\phi)$. This gives a representation of $\Spin(V,\phi)$ to $\GL_{2^{n-1}}(\bC)$ and  we identify $\Spin(V,\phi)$ with its image in this representation.

We now turn to the case of $(V,\phi) = (V_{m,n},q_{m,n})$. Let $$G_{m,n} = \Spin(V_{m,n},q_{m,n})\quad\text{and}\quad C_{m,n}= C(V_{m,n},q_{m,n}).$$ We choose the basis  $e_1,\ldots,e_m,e_{m+1},\ldots e_{m+n}$ of $V_{m,n}$ that satisfy 
\begin{equation}
q_{m,n}(e_i) = \left\{
\begin{array}{rl}
1 & \text{if } 1\leq i \leq m\\
-1 & \text{if } m+1\leq i \leq m+n
\end{array} \right.
\label{eq:normal}
\end{equation}
equwhich is normal. 
Using this basis we get by the above construction a specific faithful irreducible representation of $G_{m,n}$ to  $\GL(C_{m,n}^0)\cong \GL_{2^{m+n-1}}(\bC)$, by acting as right multiplication. The isomorphism $\GL(C_{m,n}^0)\cong \GL_{2^{m+n-1}}(\bC)$ depends on our specific choice of a normal basis and we fix this choice throughout. For any ring $R\subset \bC$, we let $G_{m,n}(R):= G_{m,n}\cap \GL_{2^{m+n-1}}(R)$ and call this group the $R$-points of $G_{m,n}$. We remark that the (conjugacy class of the) representation of $G_{m,n}$ is independent of a choice of normal basis for the quadratic space, but the group of $R$-point, for a general ring $R$, may depend on this choice. For this reason, we \emph{fix} through out the above representations of $G_{m,n}$. 

The group $G_{m,n}$ is known to be an almost simple and absolutely simple algebraic group defined over $\bQ$.  For any field $k\subset \bR$,  $\text{k-rank}(G_{m,n}) = \text{min}(m,n)$ (See for example, \cite[\S 20]{FulHar91}).  Moreover, the following follows from Corollary \ref{cor:quadratics equivalence}:

\begin{cor}\label{cor:local isomorphism}
Let $m>4, n>0$ be natural numbers, $k$ be any number field and $v$ a discrete valuation on $k$. Then, under the fixed representations which are described above, $G_{m,n}(\cO_v)$ is isomorphic to $G_{m-4,n+4}(\cO_v)$, where $\cO_v$ denotes the ring of integers in the completion of $k$ with respect to $v$.
\end{cor}
\begin{proof}
Using Corollary \ref{cor:quadratics equivalence}, this follows readily from the definitions in \cite[Chapter 10 \S2]{cassels:RQF}. Moreover, if we would have defined these groups in the language of groups schemes, This corollary was immediate. Nevertheless, since this is crucial for our construction, we give a complete proof, which is rather technical.  

Let $m' = m-4, n'=n+4$, and let $k_v$ denote the completion of $k$ with respect to $v$. Let $(k_v^{m+n},q_{m,n})$ be the quadratic space associated $q_{m,n}$ with the standard normal basis $e_1,\ldots,e_{m+n}$, that is,
\begin{equation}
q_{m,n}(e_i) = \left\{
\begin{array}{rl}
1 & \text{if } 1\leq i \leq m\\
-1 & \text{if } m+1\leq i \leq m+n.
\end{array} \right.
\label{fI1}
\end{equation}
By Corollary \ref{cor:quadratics equivalence}, there exists $M\in \GL_{m+n}(\bZ_p)$ such that $\{f_i:= Me_i\}_{i=1}^{m+n}$ is also a normal basis of $(k_v^{m+n},q_{m,n})$ with 
\begin{equation}
q_{m,n}(f_i) = \left\{
\begin{array}{rl}
1 & \text{if } 1\leq i \leq m'\\
-1 & \text{if } m'+1\leq i \leq m'+n'.
\end{array} \right.
\label{fI2}
\end{equation}
Let $C:= C^0(k_v^{m+n},q_{m,n})$ and $C' = C^0(k_v^{m'+n'},q_{m',n'})$. We show that $\GL(C)(\cO_v)\cong \GL(C')(\cO_v)$ and the corollary readily follows since $G_{m,n}(\cO_v):= G_{m,n}\cap \GL(C)(\cO_v)$ (and resp.\ for $G_{m',n'}$).

The construction above show that the bases $\{e_i\}$ and $\{f_i\}$ give rise to different bases on $C$, which we denote by $E$ and $F$. The bases $E$ and $F$ gives rise to two structures on $\GL(C)$, i.e., two isomorphisms $$\Phi_E,\Phi_F:\GL(C) \ra \GL_{2^{m+n-1}}(\bC).$$ The base $E$ is the base fixed above and by $\GL(C)(\cO_v)$ we mean the $\cO_v$-points of $\GL(C)$ with respect to the basis $E$, i.e., $\Phi_E^{-1}(\GL_{2^{m+n}}(\cO_v))$. By equation \ref{fI2} above, $\GL(C')(\cO_v)$ is isomorphic to the $\cO_v$-points of $\GL(C)$ with respect to the basis $F$. So we may conclude by showing that the $\cO_v$-points of $\GL(C)$ with respect to $E$ are isomorphic to the $\cO_v$-points of $\GL(C)$ with respect to $F$.

Let $\tilde M$ denote the  base change matrix from $E$ to $F$, which is called the \emph{derived matrix} of $M$. From the multiplication rules in $C$ and the fact that $M$ has entries in $\bZ_p$ and $q_{m,n}$ and $q_{m',n'}$ has coefficients in $\bZ_p$, it follows that $\tilde M$ also has entries in $\bZ_p$. The inverse of $\tilde M$ is the derived matrix of $M^{-1}$, which also has entries in $\bZ_p$ by the same argument. Conjugation by $\tilde M$, which we denote by $\rm Int(\tilde M)$, identify the representations $\Phi_E,\Phi_F$, that is, $\rm Int(\tilde M)\circ \Phi_E = \Phi_F$. It follows that the $\cO_v$-points of $\GL(C)$ with respect to $E$ and with respect to $F$ are isomorphic by $\rm Int(\tilde{M})$.

\end{proof}
\subsection{A basic Lemma on profinite completion}
\begin{lem}\label{lem:subgroups of profinite completion}
Let $\Gamma$ be a residually finite group. There is a one-to-one correspondence between the set $\cX$ of all finite-index subgroups of $\Gamma$  and the set $\cY$ of all open subgroup of $\widehat \Gamma$, given by $$X \mapsto \bar X\quad (X \in \cX)$$ $$Y \mapsto Y\cap \Gamma \quad (Y \in \cY)$$
where $\bar X$ denotes the closure of $X$ in $\widehat \Gamma$. Moreover, $\bar X$ is canonically isomorphic to $\widehat X$ and  $$[\Gamma:X] = [\widehat \Gamma: \widehat X].$$
\end{lem}
\begin{proof}
See \cite[Proposition 16.4.3]{LubotzkySegal}
\end{proof}
\section{Property $(T)$ in not profinite - proof of Theorem \ref{thm:main theorem}}
We continue with the notation of Theorem \ref{thm:main theorem} and we let $G_1:= G_{n,1}$ and $G_2:= G_{n-4,5}$ with the fixed representations described above. Let $\sigma_1,\sigma_2$ be the two distinct embeddings of $k$ into $\bR$. They induce natural embeddings $\hat\sigma_1,\hat\sigma_2$ of $G_i(\cO_k)$ to $G_i(\bR)$. We embed  $$\Gamma:=G_1(\cO_k) \ra G_1(\bR)\times G_1(\bR),\, x\mapsto (\hat\sigma_1(x),\hat\sigma_2(x)),$$ and similarly
$$\Lambda:=G_2(\cO_k) \ra G_2(\bR)\times G_2(\bR),\, x\mapsto (\hat\sigma_1(x),\hat\sigma_2(x))$$
It is well-known that these embeddings realize $\Gamma$ and $\Lambda$ as irreducible lattices.

We now show that any finite-index subgroup  $\Gamma_0 < \Gamma$ does not have property $(T)$ while any finite-index subgroup  $\Lambda_0 < \Lambda$ does. The group $G_1(\bR) = \Spin(n,1)(\bR)$, which is a central extension of $\SO(n,1)(\bR)$, does not have property (T) (\cite[Theorem 3.5.4]{BHV:(T)}) and neither does the direct product of $G_1(\bR)$ with itself.  (\cite[Proposition 1.7.8]{BHV:(T)}). Note that any finite-index subgroup  $\Gamma_0 < \Gamma$ is a lattice in $G_1(\bR)\times G_1(\bR)$. Since a lattice in a group has property (T) if and only if the group has property (T) (\cite[Proposition 1.7.1]{BHV:(T)}), it follows that any finite-index subgroup  $\Gamma_0 < \Gamma$ does not have property (T).

In contrast, the group $G_2(\bR) = \Spin(n-4,5)(\bR)$, which is a central extension of $\SO(n-4,5)(\bR)$, does have property (T) since it is almost simple of rank $\geq 2$ and therefore so does its direct product with itself (\cite[Proposition 1.7.8]{BHV:(T)}). Again, any finite-index subgroup  $\Lambda_0 < \Lambda$ is a lattice in $G_2(\bR)\times G_2(\bR)$. It follows that any finite-index subgroup  $\Lambda_0 < \Lambda$  has property (T). In particular, a finite-index subgroup of $\Gamma$ cannot be isomorphic to a finite-index subgroup of $\Lambda$.

Nevertheless, we will now show that there exist a finite-index subgroup $\Gamma_0<\Gamma$ and a finite-index subgroup $\Lambda_0 < \Lambda$ that have isomorphic profinite completions. First note that by Corollary  \ref{cor:local isomorphism}, $G_1(\cO_v) \cong G_2(\cO_v)$ for any discrete valuation $v$, so there exists an isomorphism $$\Phi:\prod_v G_1(\cO_v) \ra \prod_v G_2(\cO_v)$$
where the product runs over all the discrete valuations of $k$.

By \cite[11.3]{Kneser79} the congruence kernel of $G_2$ is the trivial group and therefore $$\widehat \Lambda = \widehat{ G_2(\cO_k)}   \cong  G_2(\widehat \cO_k) = \prod_v G_2(\cO_v)$$
where the product runs over all discrete valuations of $k$ . For $G_1$, by \cite[11.5c]{Kneser79} the congruence kernel is of size 1 or 2. In any case, it is finite, so $\widehat \Gamma$ fits into the following short exact sequence $$1 \ra C \ra \widehat \Gamma \ra  \prod_v G_1( \cO_v)\ra 1$$ where $C$ is a finite group.
As $\widehat \Gamma$ is profinite, we can find a finite-index open subgroup that intersects $C$ trivially and therefore maps to $\prod_v G_1( \cO_v)$ injectively.  By Lemma \ref{lem:subgroups of profinite completion}, this subgroup  is necessarily of the form $\widehat{\Gamma_0}$, for a finite-index subgroup $\Gamma_0$ of $\Gamma$. 
We identify $\widehat{\Gamma_0}$ with its (faithful) image in $\prod_v G_1( \cO_v)$. By Lemma \ref{lem:subgroups of profinite completion}, there exists a finite-index subgroup $\Lambda_0$ of $\Lambda$ with $\widehat{\Lambda_0} \cong \Phi(\widehat{\Gamma_0})$, i.e., $\widehat{\Lambda_0}$ fits in the following commutative diagram:
$$\xymatrix{\widehat{\Gamma_0}\ar@{^{(}->}[rr]\ar[d]^{\Phi|_{\widehat{\Gamma_0}}}& &\widehat\Gamma\ar[d]^{\Phi}\\ \widehat{\Lambda_0}\ar@{^{(}->}[rr]& &\widehat\Lambda}$$

 Thus $\widehat{\Lambda_0}$ and $\widehat{\Gamma_0}$ are isomorphic. As explained above,  $\Lambda_0$ has property (T) and $\Gamma_0$ does not. This concludes the proof of Theorem \ref{thm:main theorem}.
\subsection{Remarks}\label{rank remark}
The above proof also show that the rank of the Lie group containing $\Gamma$ as a lattice is not a profinite property. Indeed, using Corollary \ref{cor:local isomorphism} and induction, one sees that for arbitrarily large fixed $n$, and every $0\leq k\leq \frac{n}{8}$ the groups $G_k\times G_k$ with $G_k:=\Spin(4k+1,n-4k)$ have irreducible lattice $\Gamma_k$ which share the same profinite completions, but are pairwise non-isomorphic, by Mostow's strong-rigidty (which may be applied to the images of $\Gamma_k$ in $\SO(4k+1,n-4k)\times \SO(4k+1,n-4k)$). The rank of $G_k\times G_k$ is $2(4k+1)$.


\bibliography{biblio}
\end{document}